\def\draft{1}  
\documentclass[11pt]{article}

\usepackage[dvips]{graphicx}%
\usepackage{subfigure}
\usepackage{graphicx}
\usepackage{amsmath}
\usepackage{mathrsfs}
\usepackage{amssymb}
\usepackage{amsthm}
\usepackage{url}
\usepackage{caption}
\usepackage{caption}

\usepackage{enumitem}
\setlist{nolistsep}

\usepackage{pstricks,pst-plot}
\usepackage{pst-tree}

\usepackage{clrscode}

\ifnum\draft=1
\newcommand{\Znote}[1]{{[\bf Zhenming's Note: #1]}}
\newcommand{\Knote}[1]{{[\bf Kai-Min's Note: #1]}}
\newcommand{\Mnote}[1]{{[\bf Michael's Note: #1]}}
\newcommand{\Hnote}[1]{{[\bf Henry's Note: #1]}}
\else
\newcommand{\Znote}[1]{{}}
\newcommand{\Knote}[1]{{}}
\newcommand{\Mnote}[1]{{}}
\newcommand{\Hnote}[1]{{}}
\fi

\newtheorem{claim}{Claim}[section]%
\newtheorem{theorem}{Theorem}[section]

\newtheorem{lemma}[theorem]{Lemma}

\newtheorem{definition}[theorem]{Definition}
\newenvironment{example}[1][Example]{\begin{trivlist}
\item[\hskip \labelsep {\bfseries #1}]}{\end{trivlist}}
\newenvironment{remark}[1][Remark]{\begin{trivlist}
\item[\hskip \labelsep {\bfseries #1}]}{\end{trivlist}}

\newcommand{\E}{\mathrm{E}}

\newcommand{\mt}{{{T}}}
\newcommand{\eps}{\varepsilon}
\newcommand{\R}{{\mathbb R}}
\newcommand{\N}{{\mathbb{N}}}
\newcommand{\tM}{{\tilde{M}}}
\newcommand{\zo}{\{0,1\}}

\providecommand{\myparab}[1]{\smallskip\noindent\textbf{#1} }

\title{Chernoff-Hoeffding Bounds for Markov Chains: Generalized and Simplified}
\author{Kai-min Chung \thanks{Computer Science Department, Cornell University. Supported by a Simons Foundation Fellowship. Email: {\tt chung@cs.cornell.edu}.}
  \quad Henry Lam
\thanks{Department of Mathematics and Statistics, Boston University. Email: {\tt khlam@bu.edu}.}
\quad
Zhenming Liu \thanks{Harvard School of Engineering and Applied Sciences. Supported by  NSF grant CCF-0915922. Email: {\tt zliu@eecs.harvard.edu}.}
\quad
Michael Mitzenmacher \thanks{Harvard School of Engineering and Applied Sciences. Supported in part by NSF grants CCF-0915922 and IIS-0964473. Email: {\tt michaelm@eecs.harvard.edu}.} 
}
\date{\today}                                           
\begin{document}

\begin{titlepage}
\maketitle

\begin{abstract}

We prove the first Chernoff-Hoeffding bounds for \emph{general} (\emph{irreversible}) finite-state Markov chains based on the standard $L_1$ (variation distance) \emph{mixing-time} of the chain. Specifically, consider an ergodic Markov chain $M$ and a weight function $f: [n] \rightarrow [0,1]$ on the state space $[n]$ of $M$ with mean $\mu \triangleq \E_{v\leftarrow \pi}[f(v)]$, where $\pi$ is the stationary distribution of $M$. A $t$-step random walk $(v_1,\dots,v_t)$ on $M$ starting from the stationary distribution $\pi$ has expected total weight $\E[X] = \mu t$, where $X \triangleq \sum_{i=1}^t f(v_i)$. Let $\mt$ be the $L_1$ mixing-time of $M$. We show that the probability of $X$ deviating from its mean by a multiplicative factor of $\delta$, i.e., $\Pr \left[ \left|X - \mu t\right| \geq \delta \mu t \right]$, is at most $ \exp(-\Omega\left( \delta^2  \mu  t / \mt\right))$ for $0 \leq \delta \leq 1$, and $ \exp(-\Omega\left( \delta  \mu  t / \mt\right))$ for $\delta >1$. In fact, the bounds hold even if the weight functions $f_i$'s for $i\in [t]$ are distinct, provided that all of them have the same mean $\mu$.

We also obtain a simplified proof for the Chernoff-Hoeffding bounds based on the \emph{spectral expansion} $\lambda$ of $M$, which is the square root of the second largest eigenvalue (in absolute value) of $M\tilde{M}$, where $\tilde{M}$ is the time-reversal Markov chain of $M$. We show that the probability $\Pr \left[ \left|X - \mu t\right| \geq \delta \mu t \right]$ is at most $\exp(-\Omega\left( \delta^2 (1-\lambda) \mu  t \right))$ for $0 \leq \delta \leq 1$, and $\exp(-\Omega\left( \delta (1-\lambda) \mu  t \right))$ for $\delta >1$.

Both of our results extend to continuous time Markov chains, and to the case where the walk starts from an arbitrary distribution $\varphi$, at a price of a multiplicative factor depending on the distribution $\varphi$ in the concentration bounds.

\end{abstract}

\vfill
\thispagestyle{empty}
\end{titlepage}

\section{Introduction}

In this work, we establish large deviation bounds for random walks on
general (irreversible) finite state Markov chains based on mixing
properties of the chain in both discrete and continuous time settings.
To introduce our results we focus on the discrete time setting, which
we now describe.

Let $M$ be an ergodic Markov chain with finite state space $V = [n]$
and stationary distribution $\pi$.  Let $(v_1,\dots, v_t)$ denote a
$t$-step random walk on $M$ starting from a distribution $\varphi$ on
$V$. For every $i \in [t]$, let $f_i: V \rightarrow [0,1]$ be a weight
function at step $i$ so that $\E_{v\leftarrow \pi}[f_i(v)]= \mu>0$ for
all $i$.  Define the total weight of the walk $(v_1,\dots, v_t)$ by $X
\triangleq \sum_{i=1}^t f_i(v_i)$. The expected total weight of the
random walk $(v_1,\dots, v_t)$ is $\E[\frac 1 t  X] \approx \mu$
as $t\rightarrow \infty$.

When the $v_i$'s are drawn independently according to the stationary distribution $\pi$, a standard Chernoff-Hoeffding bound says that
{\small
\begin{equation*}
\Pr\left[ \left|X - \mu t\right| \geq \delta  \mu t \right] \leq
\begin{cases}
e^{-\Omega\left( \delta^2  \mu  t\right)}  & \mbox{for $0 \leq \delta \leq 1$,}\\
e^{-\Omega\left( \delta  \mu  t \right)}  & \mbox{for $\delta > 1$.}
\end{cases}
\end{equation*}
}
However, when $(v_1,\dots,v_t)$ is a random walk on a Markov chain
$M$, it is known that the concentration bounds depend inherently on
the mixing properties of $M$, that is the speed at which a random walk
converges toward its stationary distribution.

Variants of Chernoff-Hoeffding bounds for random walk on Markov chains
have been studied in several fields with various
motivations~\cite{Gillman93,Kahale97,LP04,Lezaud04,WX05,Wagner06,Healy06}.
For instance, these bounds are linked to the performance of Markov
chain Monte Carlo integration techniques~\cite{LP04,JS97}. They
have also been applied to various online learning problem~\cite{TL10},
testing properties of a given graph~\cite{GR00}, leader election
problems~\cite{Kahale97}, analyzing the structure of the social
networks~\cite{ART10, MYK10}, understanding the performance of
data structures~\cite{FMM09}, and computational
complexity~\cite{Healy06}. Improving such bounds is therefore of general
interest.

We improve on previous work in two ways.  First, all the existing deviation bounds, as far as we
know, are based on the \emph{spectral expansion} $\lambda(M)$ of the
chain $M$. This spectral expansion $\lambda(M)$ characterizes how much $M$ can
stretch vectors in $\mathbf R^n$ under a normed space defined by the stationary distribution $\pi$, which coincides with the second
largest absolute eigenvalue of $M$ when $M$ is reversible. (A formal definition is deferred to Section~\ref{sec:prelim}.)
The most general result for Markov chains in this form (see,
e.g. \cite{Lezaud04, Wagner06}) is
{\small
\begin{equation} \label{eqn:L-result}
\Pr\left[ \left|X - \mu t\right| \geq \delta  \mu t \right] \leq
\begin{cases}
\|\varphi\|_{\pi} e^{-\Omega\left((1-\lambda) \delta^2  \mu  t\right)}  & \mbox{for $0 \leq \delta \leq 1$,}\\
\|\varphi\|_{\pi} e^{-\Omega\left((1-\lambda) \delta  \mu  t \right)}  & \mbox{for $\delta > 1$.}
\end{cases}
\end{equation}
}
where $\varphi$ is an arbitrary initial distribution and $\|\cdot\|_\pi$ is the $\pi$-norm (which we define formally later).

However, for general irreversible Markov chains, the spectral
expansion $\lambda$ does not directly characterize the mixing time of a chain and thus
may not be a suitable parameter for such bounds. A
Markov chain $M$ could mix rapidly, but have a spectral expansion
$\lambda$ close to 1, in which case Eq. (\ref{eqn:L-result}) does not
yield meaningful bound.  In fact there is a way to modify any given
Markov chain $M$ so that the modified Markov chain $M'$ has
 (asymptotically) the same mixing-time as $M$, but the spectral expansion of $M'$ equals $1$ (Appendix~\ref{asec:missing} gives a detailed construction).
It is therefore natural to seek
a Chernoff-type bound for Markov chains directly parameterized by the chain's mixing time $T$.


Second, most previous analyses for deviation bounds such as Eq. (\ref{eqn:L-result}) are based on
\emph{non-elementary} methods such as perturbation theory
\cite{Gillman93, Lezaud04, LP04, WX05}. Kahale~\cite{Kahale97} and
Healy~\cite{Healy06} provided two elementary proofs for reversible
chains, but their results yield weaker bounds than those in Eq. (\ref{eqn:L-result}). Recently, Wagner~\cite{Wagner06} provided
another elementary proof for reversible chains matching the form in
 Eq. (\ref{eqn:L-result}). Together with the technique
of ``reversiblization'' \cite{Fill91,Lezaud04}, Wagner's analysis can be
generalized to irreversible chains. However, his use of decoupling on the linear projections outright arguably leads to a loss of insight; here we provide an approach based on directly tracing the corresponding sequence of linear projections, in the spirit of~\cite{Healy06}. This more elementary approach allows us to tackle both reversible and irreversible chains in a unified manner that avoids the use of ``reversiblization".


As we describe below, we prove a Chernoff-type bound for general
irreversible Markov chains with general weight functions $f_i$ based
on the standard $L_1$ (variation distance) mixing time of the chain,
using elementary techniques based on extending ideas from
\cite{Healy06}.  The exponents of our bounds are tight up to a
constant factor.   As far as we know, this is the first
result that shows that the mixing time is sufficient to yield these
types of concentration bounds for random walks on Markov chains.
Along the way we provide a unified proof for (\ref{eqn:L-result}) for
both reversible and irreversible chains based only on elementary
analysis.  This proof may be of interest in its own right.

\section{Preliminaries} \label{sec:prelim}
Throughout this paper we shall refer $M$ as the discrete time Markov chain under consideration. Depending on the context, $M$ shall be interpreted as either the chain itself or the corresponding transition matrix (i.e. it is an $n$ by $n$ matrix such that $M_{i, j}$ represents the probability a walk at state $i$ will move to state $j$ in the next step). For the continuous time counterpart,
we write $\Lambda$ as the generator of the chain and let $M(t) = e^{t\Lambda}$, which represents the transition probability matrix from $t_0$ to $t_0 + t$ for an arbitrary $t_0$.

Let $u$ and $w$ be two distributions over the state space $\mathbf V$. The \emph{total variation} distance between $u$ and $w$ is
$\|u - w\|_{TV} = \max_{A \subseteq \mathbf V}\left|\sum_{i \in A}u_i - \sum_{i \in A}w_i\right| = \frac 1 2 ||u - w||_1.$

Let $\epsilon > 0$. The mixing time of a \emph{discrete time} Markov chain $M$ is $T(\epsilon) = \min\left\{t:\max_{x}\|xM^t - \pi\|_{TV} \leq \epsilon \right\}$, where $x$ is an arbitrary initial distribution. The mixing time of a \emph{continuous time} Markov chain specified by the generator $\Lambda$ is $T(\epsilon) = \min\left\{t:\max_{x}\|xM(t) - \pi\|_{TV} \leq \epsilon \right\}$, where $M(t) = e^{\Lambda t}$.

 We next define an inner product space specified by the stationary distribution $\pi$:
\begin{definition}[Inner product under $\pi$-kernel] Let $M$ be an ergodic Markov chain with state space $[n]$ and $\pi$ be its stationary distribution. Let $u$ and $v$ be two vectors in $R^n$. The \emph{inner product under the $\pi$-kernel} is
$\langle u, v \rangle_{\pi} = \sum_{x \in [n]}\frac{u_i v_i}{\pi(i)}.$
\end{definition}

We may verify that $\langle \cdot, \cdot \rangle_{\pi}$ indeed forms an inner product space by checking it is symmetric, linear in the first argument, and positive definite.
The $\pi$-norm of a vector $u$ in $R^n$ is $\|u\|_{\pi} = \sqrt{\langle u, u \rangle_{\pi}}$. Note that $\|\pi\|_\pi=1$. For a vector $x \in R^n$, we write $x^{\parallel} = \langle x, \pi \rangle_{\pi} \pi$ for its component along the direction of $\pi$ and $x^{\bot} = x - x^{\parallel}$ for its component perpendicular to $\pi$.

We next define the \emph{spectral norm} of a transition matrix.

\begin{definition}[Spectral norm]\label{def:spectral}Let $M$ the transition matrix of an ergodic Markov chain. Define the spectral norm of $M$ as
$\lambda(M) = \max_{\langle x, \pi \rangle_\pi = 0}\frac{\|xM\|_{\pi}}{\|x\|_{\pi}}.$
\end{definition}

When $M$ is clear from the context, we shall simply write $\lambda$ for $\lambda(M)$. We shall also refer $1-\lambda(M)$ as the \emph{spectral gap} of the chain $M$.
In the case when $M$ is reversible, $\lambda(M)$ coincides with the second largest eigenvalue of $M$ (the largest eigenvalue of $M$ is always 1). However, when $M$ is irreversible, such relation does not hold (one hint to realize that  the eigenvalues of $M$ for an irreversible chain can be complex, and the notion of being the second largest may not even be well defined). Nevertheless, we can still connect $\lambda(M)$ with an eigenvalue of a matrix related to $M$. Specifically, let $\tilde M$ be the time reversal of $M$:
$\tilde M(x, y) = \frac{\pi(y)M(y, x)}{\pi(x)}.$
The \emph{multiplicative reversiblization} $R(M)$ of $M$ is
$R(M) \equiv M \tilde M.$ The value of $\lambda(M)$ then coincides with the square root of the second largest eigenvalue of $R(M)$, i.e. $\lambda(M) = \sqrt{\lambda(R(M))}$. Finally, notice that the stationary distribution of $M$, $\tilde M$, and $R$ are all the same. These facts can be found in~\cite{Fill91}.

\section{Chernoff-Hoeffding Bounds for Discrete Time Markov Chains}
We now present our main result formally.

\begin{theorem} \label{thm:mixdeviation} Let $M$ be an ergodic Markov chain with state space $[n]$ and stationary distribution $\pi$. Let $T = T(\epsilon)$ be its $\eps$-mixing time for $\eps \leq 1/8$.  Let $(V_1,\dots, V_t)$ denote a $t$-step random walk on $M$ starting from an initial distribution $\varphi$ on $[n]$, i.e., $V_1 \leftarrow \varphi$. For every $i \in [t]$, let $f_i: [n] \rightarrow [0,1]$ be a weight function at step $i$ such that the expected weight $\E_{v\leftarrow \pi}[f_i(v)] = \mu$ for all $i$. Define the total weight of the walk $(V_1,\dots, V_t)$ by $X \triangleq \sum_{i=1}^t f_i(V_i)$. There exists some constant $c$ (which is independent of $\mu$, $\delta$ and $\epsilon$) such that
{\small
\begin{eqnarray*}
\mbox{1. } \Pr[ X \geq (1+\delta) \mu t] & \leq &
\begin{cases}
c  \|\varphi\|_\pi  \exp\left(-\delta^2   \mu t / (72T)\right) & \mbox{ for $0 \leq \delta \leq 1$} \\
c \|\varphi\|_\pi  \exp\left(-\delta   \mu t/(72T)\right) & \mbox{ for $\delta > 1$}
\end{cases}
\\
\mbox{2. }\Pr[ X \leq (1-\delta) \mu t] & \leq & c  \|\varphi\|_\pi  \exp\left(-\delta^2   \mu t/ (72T)\right) \quad  \quad \mbox{for $0 \leq \delta \leq 1$}
\end{eqnarray*}
}
\end{theorem}
\bigskip

Before we continue our analysis, we remark on some aspects of the result.
\bigskip

\myparab{Optimality of the bound} The bound given in Theorem \ref{thm:mixdeviation} is optimal among all bounds based on the mixing time of the Markov chain, in the sense that for any given $T$ and constant $\eps$, one can find a $\delta$, a family of functions $\{f_i:V\to[0,1]\}$, and a Markov chain with mixing time $T(\eps) = T$ that has deviation probabilities matching the exponents displayed in Theorem \ref{thm:mixdeviation}, up to a constant factor. In this regard, the form of our dependency on $T$ is tight for constant $\eps$.
For example, consider the following Markov chain:
\begin{itemize}
\item The chain consists of 2 states $s_1$ and $s_2$.
\item At any time step, with probability $p$ the random walk jumps to the other state and with probability $1-p$ it stays in its current state, where $p$ is determined below.
\item for all $f_i$, we have $f_i(s_1) = 1$ and $f_i(s_2) = 0$.
\end{itemize}
Notice that the stationary distribution is uniform and $T(\epsilon) = \Theta(1/p)$ when $\epsilon$ is a constant. Thus, we shall set $p = \Theta(1/T)$ so that the mixing-time $T(\eps) = T$. Let us consider a walk starting from $s_1$ for sufficiently large length $t$. The probability that the walk stays entirely in $s_1$ up to time $t$ is $(1-p)^t \approx e^{-tp} = \exp(-\Theta(t/T))$. In other words, for $\delta = 1$ we have $\Pr[X \geq (1+\delta)\mu t] =\Pr[X\geq t]=\Pr[\text{the walk stays entirely in $s_1$}]= \exp(-\Theta(t/T(\epsilon)))$. This matches the first bound in Theorem~\ref{thm:mixdeviation} asymptotically, up to a constant factor in the exponent. The second bound can be matched similarly by switching the values of $f_i(\cdot)$ on $s_1$ and $s_2$. Finally, we remark that this example only works for $\epsilon = \Omega(1)$, which is how mixing times appear in the usual contexts. It remains open, though, whether our bounds are still optimal when $\epsilon = o(1)$.
\bigskip

\myparab{Dependency on the threshold $\epsilon$ of the mixing time} Note that the dependence of $\epsilon$ only lies on $T(\epsilon)$. Since $T(\epsilon)$ is non-decreasing in $\epsilon$, it is obvious that $\epsilon=1/8$ gives the best bound in the setting of Theorem~\ref{thm:mixdeviation}. In fact, a more general form of our bound, as will be seen along our derivation later, replaces $1/72$ in the exponent by a factor $(1-\sqrt{2\epsilon})/36$. Hence the optimal choice of $\epsilon$ is the maximizer of $(1-\sqrt{2\epsilon})/T(\epsilon)$ (with $\epsilon<1/2$), which differs for different Markov chains. Such formulation seems to offer incremental improvement and so we choose to focus on the form in Theorem~\ref{thm:mixdeviation}.
\bigskip

\myparab{Comparison with spectral expansion based Chernoff bound}
The bound given in Theorem~\ref{thm:mixdeviation} is \emph{not} always stronger than spectral expansion based Chernoff bounds \eqref{eqn:L-result} that is presented in, for example, Lezaud~\cite{Lezaud04} and Wagner~\cite{Wagner06}. Consider, for instance, a
random constant degree regular graph $G$. One can see that the spectral gap
of the Markov chain induced by a random walk over $G$ is a constant with high probability. On the other hand, the mixing time of the chain is at least $\Omega(\log n)$
because the diameter of a constant degree graph is at least $\Omega(\log n)$.
Lezaud~\cite{Lezaud04} or Wagner~\cite{Wagner06} gives us a concentration bound $\Pr[X \geq (1+\epsilon)\mu t] \leq c\|\varphi\|_{\pi}\exp\left(-\Theta(\delta^2 \mu t)\right)$ when $\delta < 1$ while Theorem~\ref{thm:mixdeviation} gives us $\Pr[X \geq (1+\epsilon)\mu t] \leq c\|\varphi\|_{\pi}\exp\left(-\Theta(\delta^2 \mu t / (\log n))\right)$.
\bigskip

\myparab{Comparison with a union bound} Assuming the spectral expansion based Chernoff bound in Lezaud~\cite{Lezaud04} and Wagner~\cite{Wagner06}, there is a simpler analysis to yield
a mixing time based bound in a similar but weaker form than Theorem \ref{thm:mixdeviation}: we first divide the random walk $(V_1, ..., V_t)$ into $T(\epsilon)$ groups
for a sufficiently small $\epsilon$
such that the $i$th
group consists of the sub-walk $V_i, V_{i + T(\epsilon)}, V_{i + 2T(\epsilon)}, ...$. The walk in each group is then governed
by the Markov chain $M^{T(\epsilon)}$. This Markov chain has unit mixing time and as a result, its spectral expansion can be bounded by a constant (by using our Claim \ref{claim:mixingtospectral} below). Together with a union bound across different groups, we obtain
{\small
\begin{eqnarray}
\mbox{1. } \Pr[ X \geq (1+\delta) \mu t] & \leq &
\begin{cases}
c T  \|\varphi\|_\pi  \exp\left(-\delta^2   \mu t / (72T)\right) & \mbox{ for $0 \leq \delta \leq 1$} \\
c T  \|\varphi\|_\pi  \exp\left(-\delta   \mu t/(72T)\right) & \mbox{ for $\delta > 1$}
\end{cases} \notag
\\
\mbox{2. }\Pr[ X \leq (1-\delta) \mu t] & \leq & c  T  \|\varphi\|_\pi  \exp\left(-\delta^2   \mu t/ (72T)\right) \quad  \quad \mbox{for $0 \leq \delta \leq 1$} \label{eqn:trivial}
\end{eqnarray}}

Theorem~\ref{thm:mixdeviation} shaves off the extra leading factors of
$T$ in these inequalities, which has significant implications. For
example, Eq. (\ref{eqn:trivial}) requires the walk to be at least
$\Omega(T \log T)$, while our bounds address walk lengths between $T$
and $T \log T$. Our tighter bound further can become important when we
need a tighter polynomial tail bound.

As a specific example, saving the factor of $T$ becomes significant when we generalize these bounds to continuous-time chains using the  discretization strategy in Fill~\cite{Fill91} and Lezaud~\cite{Lezaud04}. The strategy is to apply known discrete time bound on the discretized continuous time chain, say in a scale of $b$ units of time, followed by taking limit as $b\to0$ to yield the corresponding continuous time bound. Using this to obtain a continuous analog of Eq. (\ref{eqn:trivial}) does not work, since under the $b$-scaled discretization the mixing time becomes $T/b$, which implies that the leading factor in Eq. (\ref{eqn:trivial}) goes to infinity in the limit as $b\to0$.
\bigskip


We now proceed to prove Theorem~\ref{thm:mixdeviation}.

\begin{proof} (of Theorem~\ref{thm:mixdeviation})
We partition the walk $V_1, ..., V_t$ into $T=T(\epsilon)$ subgroups so that the $i$-th sub-walk consists of the steps $(V_i, V_{i + T}, ...)$. These sub-walks can be viewed as generated from Markov chain $N \triangleq M^T$. Also, denote $X^{(i)} \triangleq \sum_{0 \leq j \leq t / T}f_{i + jT}(V_{i + jT})$ as the total weight for each sub-walk
and $\bar X = \sum_{i = 1}^T X^{(i)}/T$ as the average total weight.

Next, we follow Hoeffding's approach~\cite{Hoeffding63} to cope with the correlation among the $X^{(i)}$. To start,
{\small
\begin{equation}
\Pr[X \geq (1+\delta)\mu t] = \Pr\left[\bar X \geq (1+\delta) \frac{\mu t}{T} \right] \leq \frac{\E[e^{r\bar X}]}{e^{r(1+\delta)\mu t/T}}. \label{eqn:intermediate1}
\end{equation}
}
Now noting that $\exp(\cdot)$ is a convex function, we use Jensen's inequality to obtain
{\small
\begin{equation}
\E[e^{r\bar X}] \leq \sum_{i \leq T}\frac{1}{T}\E[e^{rX^{(i)}}]. \label{eqn:intermediate2}
\end{equation}}

We shall focus on giving an upper bound on $\E[e^{rX^{(i)}}]$. This requires two steps:
\begin{itemize}
\item First, we show the chain $N$ has a constant spectral gap based on the fact that it takes one step to mix.
\item Second, we appy a bound on the moment generating function of $X^{(k)}$ using its spectral expansion.
\end{itemize}
Specifically, we shall prove the following claims, whose proofs will be deferred to the next two subsections.

\begin{claim}\label{claim:mixingtospectral}Let $M$ be a general ergodic Markov chain with $\epsilon$-mixing time $T(\epsilon)$. We have
$\lambda(M^{T(\epsilon)}) \leq \sqrt{2\epsilon}$.
\end{claim}

\begin{claim}\label{claim:spectralchernoff} Let $M$ be an ergodic Markov chain with state space $[n]$, stationary distribution $\pi$, and spectral expansion $\lambda = \lambda(M)$.  Let $(V_1,\dots, V_t)$ denote a $t$-step random walk on $M$ starting from an initial distribution $\varphi$ on $[n]$, i.e., $V_1 \leftarrow \varphi$.  For every $i \in [t]$, let $f_i: [n] \rightarrow [0,1]$ be a weight function at step $i$ such that the expected weight $\E_{v\leftarrow \pi}[f_i(v)] = \mu$ for all $i$. Define the total weight of the walk $(V_1,\dots, V_t)$ by $X \triangleq \sum_{i=1}^t f_i(V_i)$.  There exists some constant $c$ and a parameter $r >0$ that depends only on $\lambda$ and $\delta$ such that
{\small
\begin{eqnarray*}
\mbox{1. } \frac{\E[e^{rX}]}{e^{r(1+\delta)\mu t}} & \leq &
\begin{cases}
c\|\varphi\|_\pi \exp\left(-\delta^2 \ (1-\lambda)  \mu t / 36\right) & \mbox{ for $0 \leq \delta \leq 1$} \\
c\|\varphi\|_\pi  \exp\left(-\delta  (1-\lambda)  \mu t/36\right) & \mbox{ for $\delta > 1$.}
\end{cases}
\\
\mbox{2. }\frac{\E[e^{-rX}]}{e^{-r(1-\delta)\mu t}} & \leq & c \|\varphi\|_\pi  \exp\left(-\delta^2  (1-\lambda)  \mu t/ 36\right) \quad  \quad \mbox{for $0 \leq \delta \leq 1$.}
\end{eqnarray*}
}
\end{claim}


Claim~\ref{claim:mixingtospectral} gives a bound on the spectral expansion of each sub-walk $X^{(i)}$, utilizing the fact that they have unit mixing times. Claim~\ref{claim:spectralchernoff} is a spectral version of Chernoff bounds for Markov chains.  As stated previously, while similar results exist, we provide
our own elementary proof of claim~\ref{claim:spectralchernoff}, both for completeness and because it may be of independent interest.

We now continue the proof assuming these two claims. Using Claim~\ref{claim:mixingtospectral}, we know $\lambda(N) \leq \frac 1 2$. Next,
by Claim~\ref{claim:spectralchernoff}, for the $i$-th sub-walk, we have
{\small
\begin{equation}\label{eqn:bound1}
\frac{\E[e^{rX^{(i)}}]}{e^{r(1+\delta)\mu t/T}}\leq\begin{cases}
c\|\varphi M^i\|_\pi  \exp\left(-\delta^2  \mu t / (72T)\right) & \mbox{ for $0 \leq \delta \leq 1$} \\
c\|\varphi M^i\|_\pi  \exp\left(-\delta    \mu t/(72T)\right) & \mbox{ for $\delta > 1$}
\end{cases}
\end{equation}
}
for an appropriately chosen $r$ (which depends only on $\lambda$ and $\delta$ and hence the same for all $i$). Note that $M^i$ arises because $X^{(i)}$ starts from the distribution $\varphi M^i$. On the other hand, notice that $\|\varphi M^i\|_{\pi}^2=\|\varphi^\parallel M^i\|_{\pi}^2+\|\varphi^\bot M^i\|_{\pi}^2\leq \|\varphi^\parallel\|_{\pi}^2+\lambda^2(M^i)\|\varphi^\bot\|_{\pi}^2\leq\|\varphi\|_{\pi}^2$ (by using Lemma~\ref{lem:M-operator}), or in other words $\|\varphi M^i\|_{\pi}\leq\|\varphi\|_{\pi}$.
Together with \eqref{eqn:intermediate1} and \eqref{eqn:intermediate2}, we obtain
{\small
$$\Pr[ X \geq (1+\delta) \mu t]  \leq
\begin{cases}
c  \|\varphi\|_\pi  \exp\left(-\delta^2  \mu t / (72T)\right) & \mbox{ for $0 \leq \delta \leq 1$} \\
c \|\varphi\|_\pi  \exp\left(-\delta   \mu t/(72T)\right) & \mbox{ for $\delta > 1$}
\end{cases}$$}
This proves the first half of the theorem. The second case can be proved in a similar manner, namely that
$$
\Pr[X\leq(1-\delta)\mu t]=\Pr\left[\bar{X}\leq\frac{(1-\delta)\mu t}{T}\right]
\leq\frac{\E[e^{-r\bar{X}}]}{e^{-r(1-\delta)\mu t/T}}\leq\sum_{k=1}^{T}\frac{1}{T}\frac{\E[e^{-rX^{(k)}}]}{e^{-r(1-\delta)\mu t/T}}
\leq c  \|\varphi\|_\pi  \exp\left(-\delta^2  \mu t/ (72T)\right)
$$
again by Jensen's inequality applied to $\exp(\cdot)$.

\end{proof}

\subsection{Mixing Time v.s. Spectral Expansion}
In this subsection we prove Claim~\ref{claim:mixingtospectral}. We remark that Sinclair~\cite{Sinclair92} presents a similar result for reversible Markov chains: for every parameter $\eps \in (0,1)$,
{\small
\begin{equation} \label{eqn:sinclair-bound}
\frac{1}{2}  \frac{\lambda(M)}{1-\lambda(M)} \log \frac{1}{2\eps} \leq \mt(\eps),
\end{equation}
}
where $\mt(\eps)$ is the $\eps$-mixing-time of $M$.
However, in general it is impossible to get a bound on $\lambda(M)$ based on mixing time information for general irreversible chains because a chain $M$ can have $\lambda(M) = 1$ but the $\eps$-mixing-time of $M$ is, say, $\mt(\eps) = 2$ for some constant $\epsilon$ (and $\lambda(M^2)\ll 1$).

In light of this issue, our proof of Claim~\ref{claim:mixingtospectral} depends crucially on the fact that $M^{\mt(\eps)}$ has mixing time 1, which, as we shall see, translates to a bound on its spectral expansion that holds regardless of reversibility.
We need the following result on reversible Makrov chains, which is stronger result than Eq. (\ref{eqn:sinclair-bound}) from \cite{Sinclair92}.

\begin{lemma}\label{mixing-expansion-reversible} Let $0 < \eps \leq 1/2$ be a parameter. Let $M$ be an ergodic reversible Markov chain with $\eps$-mixing time $\mt(\eps)$ and spectral expansion $\lambda(M)$. It holds that
$ \lambda(M) \leq (2\eps)^{1/\mt(\eps)}.$
\end{lemma}

We remark that it appears possible to prove Lemma~\ref{mixing-expansion-reversible} by adopting an analysis similar to Aldous'~\cite{Aldous82}, who addressed the continuous time case.  We present an alternative proof that is arguably simpler; in particular, our proof does not use the spectral representation theorem as used in~\cite{Aldous82} and does not involve arguments that take the number of steps to infinity.

\begin{proof} (of Lemma~\ref{mixing-expansion-reversible}) Recall that for an ergodic reversible Markov chain $M$, it holds that $\lambda(M^t) = \lambda^t(M)$ for every $t \in \N$. Hence, it suffices to show that $\lambda(M^{T(\epsilon)}) \leq 2\epsilon$. Also, recall that $\lambda(M^{T(\epsilon)})$ is simply the second largest eigenvalue (in absolute value) of $M^{T(\epsilon)}$. Let $v$ be the corresponding eigenvector, i.e. $v$ satisfies $vM^{T(\epsilon)} = \lambda(M^{T(\epsilon)}) v$. Since $M$ is reversible, the entries of $v$ are real-valued. Also, notice that $v$ is a left eigenvector of $M$ while $(1, 1, ..., 1)^T$ is a right eigenvector of $M$ (using the fact that each row of $M$ sums to one). Furthermore, $v$ and $(1, ..., 1)^T$ do not share the same eigenvalue. So we have $\langle v, (1, ..., 1)^T \rangle = 0$
, i.e. $\sum_{i}v_i = 0$.
Therefore, by scaling $v$, we can assume w.l.o.g. that $x \triangleq v + \pi$ is a distribution. We have the following claim.

\begin{claim}\label{claim:mixspeedup} Let $x$ be an arbitrary initial distribution. Let $M$ be an ergodic Markov chain with stationary distribution $\pi$ and mixing time $T(\epsilon)$. We have
$\|xM^{T(\epsilon)} - \pi \|_{TV} \leq 2\epsilon \|x - \pi \|_{TV}.$
\end{claim}

\begin{proof} (of Claim~\ref{claim:mixspeedup}) The inequality holds trivially when $x=\pi$. Let $x\neq\pi$ be an arbitrary distribution on $M$, $\delta \triangleq \|x - \pi \|_{TV}>0$,
and $y \triangleq x - \pi$. We decompose $y$ into a positive component $y^+$ and a negative component $y^-$ by
{\small
\begin{equation*}
y^+_i =
\begin{cases}
y_i & \mbox{if $y_i \geq 0$} \\
0 & \mbox{o.w.}
\end{cases}
\mbox{ and } \quad
y^-_i =
\begin{cases}
0 & \mbox{if $y_i \geq 0$} \\
-y_i & \mbox{o.w.}
\end{cases}
\end{equation*}
}
Note that by definition, $\sum_i y^+_i = \sum_i y^-_i = \delta$. We define $z^+ = y^+ / \delta$ and $z^- = y^- / \delta$.  Observe that $z^+$ and $z^-$ are distributions. By the definition of $\eps$-mixing time, we have
$$ \| z^+ M^{\mt(\eps)} - \pi \|_{TV} \leq \eps, \mbox{ and }\quad \| z^- M^{\mt(\eps)} - \pi \|_{TV} \leq \eps,$$
or equivalently,
$ \| z^+ M^{\mt(\eps)} - \pi \|_1 \leq 2\eps$  and  $\| z^- M^{\mt(\eps)} - \pi \|_1 \leq 2\eps$.
Now, we are ready to bound the statistical distance $\|xM^{\mt(\eps)} - \pi \|_{TV}$ as follows.
\begin{eqnarray*}
\|xM^{\mt(\eps)} - \pi \|_{TV} & = & (1/2) \|xM^{\mt(\eps)} - \pi \|_{1}  \\
& = & (1/2) \|(x - \pi )M^{\mt(\eps)} \|_{1} \\
& = & (1/2) \|(y^+ - y^-)M^{\mt(\eps)} \|_{1} \\
& = & (1/2) \|\delta z^+M^{\mt(\eps)} - \delta z^-M^{\mt(\eps)} \|_{1} \\
& = & (\delta/2) \|(z^+M^{\mt(\eps)} - \pi) - ( z^-M^{\mt(\eps)} - \pi)\|_{1} \\
& \leq & (\delta/2) \left( \|(z^+M^{\mt(\eps)} - \pi) \|_1  + \| ( z^-M^{\mt(\eps)} - \pi)\|_{1} \right)
 \leq  2\eps\delta.
\end{eqnarray*}
\end{proof}
We now continue to prove Lemma~\ref{mixing-expansion-reversible}. By Claim~\ref{claim:mixspeedup},
$ \|xM^{\mt(\eps)} - \pi \|_{TV} \leq 2\eps  \|x - \pi \|_{TV}$, i.e.
$\|xM^{\mt(\eps)} - \pi \|_{1} \leq 2\eps  \|x - \pi \|_{1}.$
Observing that $(xM^{\mt(\eps)} - \pi)$ and $(x - \pi)$ are simply $\lambda(M^{\mt(\eps)}) v$ and $v$, the above inequality means
$\lambda(M^{\mt(\eps)}) \|v\|_1 \leq 2\eps  \|v \|_1,$
which implies $\lambda(M^{\mt(\eps)}) \leq 2\eps,$ as desired.
\end{proof}

We are now ready to prove our main claim.

\begin{proof}{(of Claim~\ref{claim:mixingtospectral})} The idea is to reduce to the reversible case by considering the reversiblization of $M^{\mt(\eps)}$. Let $\tM^{\mt(\eps)}$ be the time reversal of $M^{\mt(\eps)}$, and $R \triangleq M^{\mt(\eps)}\tM^{\mt(\eps)}$ be the reversiblization of $M^{\mt(\eps)}$. By Claim~\ref{claim:mixingtospectral},
$\lambda(M^{\mt(\eps)}) = \sqrt{\lambda(R)}.$
Let us recall (from Section~\ref{sec:prelim}) that $M$, $M^{T(\eps)}$, and $\tM^{\mt(\eps)}$ all share the same stationary distribution $\pi$.
Next, we claim that the $\eps$-mixing-time of $R$ is $1$. This is because $\|\varphi M^{T(\eps)}\tM^{\mt(\eps)}-\pi\|_{TV}\leq \|\varphi M^{T(\eps)}-\pi\|_{TV}\leq\eps$, where the second inequality uses the definition of $T(\eps)$ and the first inequality holds since any Markov transition is a contraction mapping: for any Markov transition, say $S=(s(i,j))$, and any vector $x$, $\|xS\|_1=\sum_j|\sum_ix_is(i,j)|\leq\sum_j\sum_i|x_i|s(i,j)=\sum_i|x_i|=\|x\|_1$; putting $x=\varphi M^{T(\eps)}-\pi$ and $S=\tM^{\mt(\eps)}$ gives the first inequality.
 Now, by Lemma~\ref{mixing-expansion-reversible}, $\lambda(R) \leq 2\eps$, and hence
$\lambda(M^{\mt(\eps)}) = \sqrt{\lambda(R)} \leq \sqrt{2\eps},$
as desired.
\end{proof}

\subsection{Bounding the Moment Generating Function}\label{sec:boundmgf}
We now prove Claim~\ref{claim:spectralchernoff}. We focus on the first inequality in the claim; the derivation of the second inequality is similar and is deferred to Appendix~\ref{asec:less}.

Claim~\ref{claim:spectralchernoff} leads directly to a spectral version of the
Chernoff bound for Markov chains. Lezaud~\cite{Lezaud04} and
Wagner~\cite{Wagner06} give similar results for the case where $f_i$
are the same for all $i$. The analysis of~\cite{Wagner06} in particular can be extended
to the case where the functions $f_i$ are different. Here we
present an alternative analysis and along the way will discuss the merit of our approach compared to the previous proofs.

Recall that we define $X = \sum_{i = 1}^tf_i(V_i)$. We start with the following observation, which
has been used previously \cite{Healy06, Lezaud04, Wagner06}.

\begin{equation}\label{eqn:walkprob}
\E[e^{rX}] = \|\varphi P_1 M P_2 ... M P_t\|_1,
\end{equation}
where the $P_i$ are diagonal matrices with diagonal entries $(P_i)_{j,j} \triangleq e^{rf_i(j)}$ for $j \in [n]$. One can verify this fact by observing that each walk $V_1,\ldots,V_t$ is assigned the corresponding probability in the product of $M$'s with the appropriate weight $e^{r\sum_if_i(V_i)}$.

For ease of exposition, let us assume $P_i$ are all the same at this moment. Let $P = P_1 = ... = P_t$, then (\ref{eqn:walkprob}) becomes $\|\varphi (PM)^{t - 1}P\|_1 = \langle \varphi (PM)^{t-1}P, \pi \rangle_{\pi}= \langle \varphi (PM)^t, \pi \rangle_{\pi}=\|\varphi (PM)^t\|_1$ (see Lemma~\ref{lem:M-operator} below). Up to this point, our analysis is similar to previous work \cite{Gillman93, Lezaud04, Healy06, Wagner06}. Now there are two natural possible ways of bounding $\|\varphi (PM)^t\|_1=\langle \varphi (PM)^t, \pi \rangle_{\pi}$.

\begin{itemize}
\item \textbf{Approach 1. Bounding the spectral norm of the matrix $PM$.} In this approach, we observe that $\langle \varphi (PM)^t, \pi \rangle_{\pi} \leq \|\varphi\|_{\pi} \|PM\|_\pi^t$ where $\|PM\|_\pi$ is the operator norm of the matrix $PM$ induced by $\|\cdot\|_\pi$ (see, for example, the proof of Theorem 1 in \cite{Wagner06}). This method decouples the effect of each $PM$ as well as the initial distribution. When $M$ is reversible, $\|PM\|_\pi$ can be bounded through Kato's spectral perturbation theory~\cite{Gillman93,Lezaud04,LP04}. Alternatively, Wagner~\cite{Wagner06} tackles the variational description of $\|PM\|_\pi$ directly, using only elementary techniques, whose analysis can be generalized to irreversible chains.

%

\item \textbf{Approach 2. Inductively giving a bound for $x(PM)^i$ for all $i \leq t$.} In this approach, we do not decouple the product $\varphi (PM)^t$. Instead, we trace the change of the vector $\varphi(PM)^i$ for each $i \leq t$. As far as we know, only Healy~\cite{Healy06} adopts this approach and his analysis is restricted to regular graphs, where the stationary distribution is uniform. His analysis also does not require perturbation theory.
\end{itemize}

Our proof here generalizes the second approach to any ergodic chains by only using elementary methods. We believe this analysis is more straightforward for the following reasons. First, directly tracing the change of the vector $\varphi(PM)^i$ for each step keeps the geometric insight that would otherwise be lost in the decoupling analysis as in~\cite{Lezaud04, Wagner06}. Second, our analysis studies both the reversible and irreversible chains in a unified manner. We \emph{do not} use the reversiblization technique to address the case for irreversible chains. While the reversiblization technique is a powerful tool to translate an irreversible Markov chain problem into a reversible chain problem, this technique operates in a blackbox manner; proofs based on this technique do not enable us to directly measure the effect of the operator $PM$.


We now continue our analysis by using a framework similar to the one presented by Healy~\cite{Healy06}. We remind the reader that we no longer assume $P_i$'s are the same. Also, recall that $\E[e^{rX}] = \|\varphi P_1MP_2...MP_t\|_1=\langle \varphi P_1MP_2...MP_t,\pi \rangle_{\pi} = \|(\varphi P_1MP_2...MP_t)^{\parallel}\|_\pi.$
Let us briefly review the strategy from \cite{Healy06}.

\begin{itemize}
\item First, we observe that an arbitrary vector $x$ in $\R^n$ can be decomposed into its \emph{parallel} component
    (with respect to $\pi$) $x^{\parallel} = \langle x, \pi \rangle \pi$ 
    and the \emph{perpendicular} component $x^{\perp} = x - x^{\parallel}$ in the $L_\pi$ space. This decomposition helps tracing the difference (in terms of the norm) between each pair of $\varphi P_1M...P_iM$ and $\varphi P_1M...P_{i + 1}M$ for $i \leq t$, i.e. two consecutive steps of the random walk. For this purpose, we need to understand the \emph{effects of the linear operators} $M$ and $P_i$ when they are applied to an arbitrary vector.
\item Second, after we compute the difference between each pair $xP_1M...P_{i}M$ and $xP_1M...P_{i + 1}M$, we set up a \emph{recursive relation}, the solution of which yields the Chernoff bound.
\end{itemize}

We now follow this step step framework to prove Claim~\ref{claim:spectralchernoff}

\myparab{The effects of the $M$ and $P_i$ operators} Our way of tracing the vector $\varphi P_1MP_2...MP_t$ relies on the following two lemmas.

\begin{lemma}{\bf(The effect of the $M$ operator)} \label{lem:M-operator} Let $M$ be an ergodic Markov chain with state space $[n]$, stationary distribution $\pi$, and spectral expansion $\lambda = \lambda(M)$.  Then
\begin{enumerate}
\item $\pi M = \pi$.
\item For every vector $y$ with $y \bot \pi$, we have $yM \bot \pi$ and $\|yM \|_\pi \leq \lambda \|y \|_\pi$.
\end{enumerate}
\end{lemma}

\begin{lemma}{\bf(The effect of the $P$ operator)} \label{lem:P-operator} Let $M$ be an ergodic Markov chain with state space $[n]$ and stationary distribution $\pi$.  Let $f:[n] \rightarrow [0,1]$ be a weight function with $\E_{v\leftarrow \pi}[f(v)] = \mu$. Let $P$ be a diagonal matrix with diagonal entries $P_{j,j} \triangleq e^{rf(j)}$ for $j \in [n]$, where $r$ is a parameter satisfying $0 \leq r \leq 1/2$. Then
\begin{enumerate}
\item $\| (\pi P)^{\parallel} \|_\pi \leq 1+(e^r - 1) \mu$.
\item $\| (\pi P)^{\bot} \|_\pi \leq 2r\sqrt{\mu}$.
\item For every vector $y \bot \pi$, $\| (y P)^{\parallel} \|_\pi \leq 2r\sqrt{\mu} \|y\|_\pi$.
\item For every vector $y \bot \pi$, $\| (y P)^{\bot} \|_\pi \leq e^r \|y\|_\pi$
\end{enumerate}
\end{lemma}

Items 1 and 4 of Lemma~\ref{lem:P-operator} state that $P$ can stretch both the perpendicular and parallel components along their original directions moderately. Specifically, a parallel vector is stretched by at most a factor of $(1+(e^r - 1) \mu) \approx 1+ O(r\mu)$ and a perpendicular vector is stretched by a factor of at most $e^r \approx 1+O(r)$. (Recall $r$ will be small.) On the other hand, items 2 and 3 of the lemma state that $P$ can create a new perpendicular component from a parallel component and vice versa, but the new component is of a much smaller size compared to the original component (i.e. only of length at most $2r\sqrt{\mu}$ times the original component).

\begin{remark}
We note that the key improvement of our analysis over that of Healy~\cite{Healy06} stems from items 2 and 3 of  Lemma~\ref{lem:P-operator}.  Healy~\cite{Healy06} proved a bound with a factor of $(e^r-1)/2 = O(r)$ for both items for the special case of undirected and regular graphs. Our quantitative improvement to $O(r\sqrt{\mu})$ (which is tight) is the key for us to prove a multiplicative Chernoff bound without any restriction on the spectral expansion of $M$.
\end{remark}

Note that Lemma~\ref{lem:M-operator} is immediate from the definitions of $\pi$ and $\lambda$. We focus on the proof of Lemma~\ref{lem:P-operator}:

\begin{proof} (of Lemma~\ref{lem:P-operator}). For the first item, note that by definition, $\| (\pi P)^{\parallel}\|_{\pi} = \langle \pi P, \pi \rangle_\pi = \sum_i e^{rf(i)} \pi_i$. We simplify the sum using the fact that $e^{rx} \leq 1 + (e^r - 1) x$ when $r, x \in [0,1]$.
{\small
$$
\| (\pi P)^{\parallel}\|_{\pi}  =  \sum_i e^{rf(i)} \pi_i
 \leq  \sum_i (1+ (e^r-1) f(i)) \pi_i
 =  \sum_i \pi_i + (e^r-1) \sum_i f(i) \pi _i
 =  1 + (e^r - 1)\mu,
$$}
where the last equality uses the fact that $\sum_i \pi_i = 1$, and $\sum_i f(i) \pi_i = \E_{v\leftarrow \pi} [f(v)] = \mu$.

For the second item, by the Pythagorean theorem , we have
{\small
$$
\|(\pi P)^{\bot}\|^2_{\pi}  =  \|\pi P\|^2_{\pi} - \|(\pi P)^{\parallel}\|^2_{\pi}
 =  \sum_{i }e^{2rf(i)}\pi_{i} - \left(\sum_{i }e^{rf(i)}\pi_i\right)^2.$$}
Recall that $r \leq 1/2$ and $f(i) \leq 1$, and therefore $2rf(i) \leq 1$. Using the fact that $1+x \leq e^x\leq 1+x+x^2$ when $x \in [0, 1]$, we have
\begin{eqnarray*}
 \sum_{i}e^{2rf(i)}\pi_{i} - \left(\sum_{i}e^{rf(i)}\pi_i\right)^2
& \leq & \sum_{i}(1+2rf(i) + 4r^2f^2(i))\pi_i - \left(\sum_{i}(1+rf(i))\pi_i\right)^2 \\
& \leq & 1 + 2r\mu + 4r^2 \mu - (1+ r\mu )^2 \\
& = & 1 + 2r\mu + 4r^2 \mu - (1+ 2 r \mu + r^2 \mu^2) \leq   4r^2 \mu,
\end{eqnarray*}
The second inequality uses the fact that $\sum_i f^2(i) \pi(i) \leq \sum_i f(i) \pi(i) = \mu$ (since $0\leq f(i) \leq 1$). It follows that
$\|(\pi P)^{\bot}\|_\pi \leq \sqrt{4r^2 \mu} = 2r \sqrt{\mu}.$

For the third item, by definition, $\|(yP)^{\parallel}\|_\pi = \langle yP, \pi \rangle_\pi$. Since $P$ is diagonal, we have $\langle yP, \pi \rangle_\pi = \langle y, \pi P\rangle_\pi$. By definition, $y \bot \pi$ means $\langle y, \pi \rangle_\pi = 0$. Therefore,
$\|(yP)^{\parallel}\|_\pi = \langle y, \pi P\rangle_\pi - \langle y, \pi \rangle_\pi = \langle y, \pi (P-I)\rangle_\pi.$
By the Cauchy-Schwarz Inequality, we have
$\langle y, \pi (P-I)\rangle_\pi \leq \|y\|_\pi \| \pi(P-I)\|_\pi.$

We proceed to upper bound $\| \pi(P-I)\|_\pi$:
{\small
$$\| \pi(P-I)\|^2_\pi  =  \sum_i (\pi_i (e^{rf(i)} - 1))^2 / \pi_i  = \sum_i (e^{rf(i)} -1)^2 \pi_i.$$}
Using $e^{rx} \leq 1 + (e^r - 1) x$ for $r, x \in [0,1]$, we have
$\sum_i (e^{rf(i)} -1)^2 \pi_i  \leq \sum_i (1 + (e^r-1) f(i) - 1)^2 \pi_i = \sum_i (e^r-1)^2 f^2(i) \pi_i \leq (2r)^2 \sum_i f(i) \pi_i \leq (2r)^2 \mu,$
where the second-to-last inequality uses the fact that $e^r -1 \leq 2r$ for $r \in [0,1]$ and $0\leq f(i)\leq 1$. Therefore,
$\|(yP)^{\parallel}\|_\pi \leq \| \pi(P-I)\|_\pi \|y\|_\pi \leq 2r \sqrt{\mu} \|y\|_\pi.$

Finally, for the fourth item, we have
$$\|(yP)^{\bot}\|^2_\pi \leq \|yP\|^2_\pi = \sum_i \frac{y_i^2 e^{2rf(i)}}{\pi_i} \leq \sum_i \frac{y_i^2 e^{2r}}{\pi_i} = e^{2r} \|y\|^2_\pi,$$
which implies $\|(yP)^{\bot}\| \leq e^r \|y\|_\pi$.
\end{proof}

\paragraph{Recursive analysis} We now provide a recursive analysis for the terms $xP_1M...MP_i$ for $i \leq t$ based on
our understanding of the effects from the linear operators $M$ and $P_i$. This completes the proof for Claim~\ref{claim:spectralchernoff}.

\begin{proof} (of Claim~\ref{claim:spectralchernoff}).
First, recall that
{\small
$$ \E[e^{rX}] = \|(\varphi P_1 M P_2 ... M P_t)^{\parallel}\|_\pi = \|(\varphi P_1 M P_2 ... M P_t M)^{\parallel}\|_\pi = \left\|\left(\varphi \prod_{i=1}^t (P_i M) \right)^{\parallel}\right\|_\pi.$$
}
where the second equality comes from Lemma~\ref{lem:M-operator}.
Our choice of $r$ is $r = \min\{1/2,\log(1/\lambda)/2, 1-\sqrt \lambda, (1-\lambda)\delta/18\}$. We shall explain how we make such a choice as we walk through our analysis.

We now trace the $\pi$-norm of both parallel and perpendicular components of the random walk for each application of $P_iM$. Let $z_0 \triangleq \varphi$ and $z_i = z_{i-1} P_iM$ for $i \in [t]$. By triangle inequality and Lemma \ref{lem:M-operator} and \ref{lem:P-operator}, for every $i \in [t]$,

\begin{eqnarray*}
\|z_i^{\parallel} \|_\pi  =  \|(z_{i-1}P_iM)^{\parallel} \|_\pi
 =  \|((z_{i-1}^{\parallel} + z_{i-1}^{\bot}) P_iM)^{\parallel} \|_\pi
& \leq &  \|(z_{i-1}^{\parallel}  P_iM)^{\parallel} \|_\pi + \|(z_{i-1}^{\bot} P_iM)^{\parallel} \|_\pi \\
& \leq & \left(1 + (e^r-1)\mu\right) \|z_{i-1}^{\parallel}\|_\pi + \left(2r \sqrt{\mu} \right) \|z_{i-1}^{\bot}\|_\pi,
\end{eqnarray*}
and similarly,
\begin{eqnarray*}
\|z_i^{\bot} \|_\pi \leq  \|(z_{i-1}^{\parallel}  P_iM)^{\bot} \|_\pi + \|(z_{i-1}^{\bot} P_iM)^{\bot} \|_\pi  & \leq  & \left(2r \lambda \sqrt{\mu}  \right)\|z_{i-1}^{\parallel}\|_\pi + \left(e^r \lambda \right) \|z_{i-1}^{\bot}\|_\pi  \\
 &\leq &  \left(2r \lambda \sqrt{\mu}  \right)\|z_{i-1}^{\parallel}\|_\pi + \sqrt{\lambda}\|z_{i-1}^{\bot}\|_\pi,
\end{eqnarray*}
where the last inequality holds when $r \leq (1/2) \log (1/\lambda)$ i.e. $e^r \leq 1/\sqrt{\lambda}$. The reason to require $r \leq (1/2) \log (1/\lambda)$ is that we can guarantee the perpendicular component is \emph{shrinking} (by a factor of $\sqrt{\lambda} < 1$) after each step.

Now let $\alpha_0 = \|z_0^{\parallel} \|_\pi=1$ and $\beta_0 = \|z_0^{\bot} \|_\pi$, and define for $i \in [t]$,
$$
\alpha_i  = \left(1 + (e^r-1)\mu\right)  \alpha_{i-1} + \left(2r \sqrt{\mu} \right) \beta_{i-1} \quad \mbox{ and } \quad
\beta_i  = \left(2r \lambda \sqrt{\mu}  \right) \alpha_{i-1} + \sqrt{\lambda} \beta_{i-1}.
$$
One can prove by induction easily that $\|z_i^{\parallel} \|_\pi \leq \alpha_i$ and $\|z_i^{\bot} \|_\pi \leq \beta_i$ for every $i\in[t]$, and $\alpha_i$'s are strictly increasing. Therefore, bounding the moment generating function $\E[e^{rX}] = \| z_t^{\parallel} \|_\pi \leq \alpha_t$ boils down to bounding the recurrence relation for $\alpha_i$ and $\beta_i$.

Observe that in the recurrence relation, only the coefficient $(1 + (e^r-1)\mu) > 1$ while the remaining coefficients $(2r\sqrt{\mu}),(2r\lambda\sqrt{\mu}),$ and $\sqrt{\lambda}$ are all less than $1$ if $r$ is chosen sufficiently small. This suggests, intuitively, $\alpha_i$'s terms will eventually dominate. This provides us a guide to reduce the recurrence relation to a single variable as follows.

First let us give an upper bound for $\beta_i$.

\begin{claim} \label{clm:bound-beta} For every $i\in [t]$,
$\beta_i \leq 2r \left( \sum_{j=0}^{i-1} \sqrt{\lambda^{j+2} \mu }  \right)  \alpha_{i-1} + \sqrt{\lambda^{i}} \beta_0.$
\end{claim}
\begin{proof} of Claim~\ref{clm:bound-beta}.
The lemma follows by expanding the recurrence relation and using the fact that $\alpha_i$'s are increasing. i.e.
{\small
$$
\beta_i  =  2r\lambda \sqrt{\mu_i} \alpha_{i-1} + \sqrt{\lambda} \beta_{i-1}
 =  2r\lambda \sqrt{\mu} \alpha_{i-1} + \sqrt{\lambda} 2r\lambda \sqrt{\mu} \alpha_{i-2} + \sqrt{\lambda^2} \beta_{i-2}
 =  \dots
 =  2r \left( \sum_{j=0}^{i-1} \sqrt{\lambda^{j+2} \mu} \alpha_{i-j-1}\right) +  \sqrt{\lambda^{i}} \beta_0
$$
}
Finally, by using the fact that $\alpha_i$ are strictly increasing, we complete the proof.
\end{proof}

We can then bound $\alpha_i$ by substituting $\beta_{i-1}$ using Claim~\ref{clm:bound-beta}.

\begin{claim} \label{clm:bound-alpha}
$\alpha_1 \leq (1 + (e^r-1)\mu)  + 2r \sqrt{\mu} \beta_0,$
and for every $2 \leq i \leq t$,
{\small
$$\alpha_i \leq \left(1 + (e^r -1) \mu + 4r^2 \sqrt{\mu} \left( \sum_{j=0}^{i-2} \sqrt{\lambda^{j+2} \mu}\right)\right) \alpha_{i-1} + 2r\sqrt{\lambda^{i-1}\mu }  \beta_0.$$}
\end{claim}
\begin{proof} The case of $i=1$ is trivial. For $2 \leq i \leq t$, this follows by applying the recurrence relation, Claim~\ref{clm:bound-beta}, and the fact that $\alpha_{i-2} < \alpha_{i-1}$.
{\small
\begin{eqnarray*}
\alpha_i & = & (1 + (e^r-1)\mu) \alpha_{i-1} + \left(2r \sqrt{\mu} \right) \beta_{i-1} \\
& \leq & (1 + (e^r-1)\mu) \alpha_{i-1} + (2r\sqrt{\mu})  \left(2r \left( \sum_{j=0}^{i-2} \sqrt{\lambda^{j+2} \mu }  \right) \alpha_{i-2} + \sqrt{\lambda^{i-1}} \beta_0 \right) \\
& \leq & \left(1 + (e^r-1)\mu +  4r^2 \sqrt{\mu} \left( \sum_{j=0}^{i-2} \sqrt{\lambda^{j+2} \mu }  \right)  \right)\alpha_{i-1} + 2r \sqrt{\lambda^{i-1}\mu} \beta_0
\end{eqnarray*}
}
\end{proof}
For notational simplicity, let $A_1=1+(e^r-1)$ and for $1<i\leq t$, let
{\small
$$A_i \triangleq \left(1 + (e^r-1)\mu +  4r^2 \sqrt{\mu} \left( \sum_{j=0}^{i-2} \sqrt{\lambda^{j+2}  \mu }  \right)  \right).$$}
 Claim~\ref{clm:bound-alpha} then can be expressed as
$\alpha_i \leq A_i \alpha_{i-1} + 2r\sqrt{\mu} \min\{ \sqrt{\lambda^{i-1}}, 1\}  \beta_0,$
 for every $i\in [t]$.
By expanding iteratively, we obtain
{\small
\begin{eqnarray*}
\alpha_t & \leq & A_t(A_{t-1}(\cdots(A_3(A_2(A_1 + 2r\sqrt{\mu} \beta_0) + 2r\sqrt{\lambda\mu}\beta_0) + 2r\sqrt{\lambda^2\mu}\beta_0)\cdots)+2r\sqrt{\lambda^{t-2}\mu}\beta_0)+2r\sqrt{\lambda^{t-1}\mu}\beta_0 \\
& = & (A_t\cdots A_1) + (A_t \cdots A_2 (2r\sqrt{\mu}\beta_0)) + (A_t \cdots A_3 (2r\sqrt{\lambda\mu}\beta_0)) + \dots + A_t (2r\sqrt{\lambda^{t-2}\mu}\beta_0) + 2r\sqrt{\lambda^{t-1}\mu}\beta_0 \\
& \leq & \left( 1+ 2r\sqrt{\mu}\beta_0 + 2r\sqrt{\lambda\mu}\beta_0 + 2r \sqrt{\lambda^2\mu} \beta_0 + \cdots 2r\sqrt{\lambda^{t-1}\mu}\beta_0\right) \left(\prod_i A_i\right)\\
& \leq &  \left(  1+ \frac{4r\sqrt{\mu}\beta_0}{1-\sqrt{\lambda}} \right) \left(\prod_i A_i\right)
 \leq  \left(  1 + \frac{8r\sqrt{\mu}\beta_0}{1-\lambda} \right) \left(\prod_i A_i\right),
\end{eqnarray*}
}
where the last inequality uses the fact that $1/(1-\sqrt{\lambda}) \leq 2/(1-\lambda)$ for $\lambda \in [0,1)$. It remains to upper bound $\prod_i A_i$. Using $(1+x) \leq e^{x}$, we have
{\small
$$\prod_{i=1}^t A_i \leq \exp\left\{ (e^r-1)\mu+\sum_{i=2}^t \left( (e^r-1)\mu_i +  4r^2 \sqrt{\mu} \left( \sum_{j=0}^{i-2} \sqrt{\lambda^{j+2} \mu }\right)\right)\right\}.$$}
The first two sums in the exponent lead to $\sum_i (e^r-1)\mu_i = (e^r-1) \mu t$. .
We now bound the last sum in the exponent, which can be viewed as an ``error'' term due to the correlation between each step of the random walk.
{\small
$$
\sum_{i=2}^t 4r^2 \sqrt{\mu} \sum_{j=0}^{i-2} \sqrt{\lambda^{j+2} \mu}
 \leq  4r^2\mu  \sum_{i=1}^t \sum_{j=0}^{i-2} \sqrt{\lambda^j} =  4r^2  \mu t  \sum_{j=0}^{t-2} \sqrt{\lambda^j}
 \leq  \frac{8r^2 \mu t}{1-\lambda},
$$
}
where last inequality uses $\sum_{j=0}^{t-2} \sqrt{\lambda^j} \leq 1/(1-\sqrt{\lambda}) \leq 2/(1-\lambda).$ Putting things together, we have
$$\prod_{i=1}^t A_i  \leq \exp \left\{ (e^r-1) \mu t + \frac{8r^2  \mu t}{1-\lambda} \right\} = \exp \left\{ \left( (e^r-1) + \frac{8 r^2  }{1-\lambda} \right)  \mu t\right\},$$
and recalling that  $\|\varphi^\parallel\|_\pi=1$ and $\beta_0 = \|\varphi^{\bot}\|_\pi $,
{\small
$$\E[e^{rX}] \leq \alpha_t \leq \left( 1 + \frac{8r\mu\beta_0}{1-\lambda} \right)  \left(\prod_i A_i\right) \leq 2\max\left\{1,\frac{8r\sqrt{\mu}}{1-\lambda}\right\}\|\varphi\|_\pi \exp\left\{\left((e^r - 1) + \frac{8 r^2 }{(1-\lambda)}\right) \mu t\right\}.$$}

Recall that our goal is to choose an $r$ to bound $\E[e^{rX}] / e^{r(1+\delta)\mu t}$. Choosing $r = \min\{1/2,\log(1/\lambda)/2, 1-\sqrt \lambda, (1-\lambda)\delta/18\} = (1-\lambda)\delta/18$, we complete the proof of Claim~\ref{claim:spectralchernoff}.
\end{proof}

Before completing this subsection, we make a final remark.
Our proof also works even for the case $\E_{\pi}[f_i(v)]$ are \emph{different} for different values of $i$, which results in a more general Chernoff type bound based on spectral expansions. This more general result, as far as we know, has not been noted in existing literatures with the exception of Healy~\cite{Healy06}, who gave a Chernoff bound of this kind with stronger assumptions for regular graphs, although the analysis given by Lezaud~\cite{Lezaud04} or Wagner~\cite{Wagner06} also appears to be generalizable as well.
On the other hand, this strengthened result of Claim~\ref{claim:spectralchernoff} does not seem to be sufficient to remove the requirement that $\E_{\pi}[f_i(v)]$ are the same for Theorem~\ref{thm:mixdeviation}.

\subsection{Continuous Time Case}
We now generalize our main result to cover the continuous time chains. The analysis is similar to the one presented by Lezaud~\cite{Lezaud04} and will be deferred to Appendix~\ref{sec:continuous time appendix}.

\begin{theorem}\label{thm:continousbound}Let $\Lambda$ be the generator of an ergodic continuous time Markov chain with state space $[n]$ and mixing time $T=T(\epsilon)$. Let $\{v_t: t \in \mathbf R^+\}$ be a random walk on the chain starting from an initial distribution $\varphi$ such that $v_t$ represents the state where the walk stay at time $t$. Let $\{f_t: [n] \rightarrow [0, 1]\mid t \in \mathbf R^+\}$ be a family of functions such that $\mu = \E_{v \leftarrow \pi}[f_t(v)]$ for all $t$. Define the weight over the walk $\{v_s: s \in \mathbf R^+\}$ up to time $t$ by $X_t \triangleq \int_0^t f_s(v_s) ds$. There exists a constant $c$ such that
{\small
\begin{eqnarray*}
\mbox{1. } \Pr[ X \geq (1+\delta)\mu t] & \leq &
\begin{cases}
c\|\varphi\|_\pi\exp\left(-\delta^2\mu t / (72T)\right) & \mbox{ for $0 \leq \delta \leq 1$} \\
c\|\varphi\|_\pi\exp\left(-\delta\mu t/(72T)\right) & \mbox{ for $\delta > 1$}
\end{cases}
\\
\mbox{2. }\Pr[ X \leq (1-\delta)\mu t] & \leq & c\|\varphi\|_\pi\exp\left(-\delta^2\mu t/ (72T)\right) \quad  \quad \mbox{for $0 \leq \delta \leq 1$}
\end{eqnarray*}}
\end{theorem}

\bibliographystyle{plain}
\bibliography{rwchernoff}

\appendix

\section{Construction of Mixing Markov Chain with No Spectral Expansion}\label{asec:missing}

In this section, we show that any ergodic Markov chain $M$ with mixing time $T = T(1/4)$ can be modified to a chain $M'$ such that $M'$ has mixing time $O(T)$ but spectral expansion $\lambda(M') = 1$.

Our modification is based on the following simple observation. Let $M'$ be an ergodic Markov chain with stationary distribution $\pi'$. If there exist two states $v$ and $v'$ such that (i) $M'_{v,v'} = 1$, i.e., state $v$ leaves to state $v'$ with probability $1$, and (ii) $M'_{u,v'} = 0$ for all $u\neq v$, i.e., the only state transits to $v'$ is $v$, then $\lambda(M') = 1$: Note that in this case, $\pi'(v) = \pi'(v')$ since all probability mass from $v$ leaves to $v'$, which receives probability mass only from $v$. Consider a distribution $x$ whose probability mass all concentrates at $v$, i.e., $x_v= 1$ and $x_u=0$ for all $u\neq v$. One step walk from $x$ results in the distribution $xM'$ whose probability mass all concentrates at $v'$. By definition, $\|x\|_{\pi'} = \|xM'\|_{\pi'}$ and thus $\lambda(M') = 1$.

Now, let $M$ be an ergodic Markov chain with mixing time $T = T(1/4)$ and stationary distribution $\pi$. We shall modify $M$ to a Markov chain $M'$ that preserves the mixing-time and satisfies the above property. We mention that it is not hard to modify $M$ to satisfy the above property. The challenge is to do so while preserving the mixing-time. Our construction is as follows.
\begin{itemize}
\item For every state $v$ in $M$, we ``split'' it into three states $(v, in), (v, mid), (v,out)$ in $M'$.
\item For every state $(v,in)$ in $M'$, we set $M'_{(v,in),(v,in)} = M'_{(v,in),(v,mid)} = 1/2$, i.e., $(v,in)$ stays in the same state with probability $1/2$ and transits to $(v,mid)$ with probability $1/2$.
\item For every state $(v,mid)$ in $M'$, we set $M'_{(v,mid),(v,out)} = 1$, i.e., $(v,mid)$ always leaves to $(v,out)$.
\item For every pairs of states $u,v$ in $M$, we set the transition probability $M'_{(u,out),(v,in)}$ from $(u,out)$ to $(v,in)$ to be $M_{u,v}$.
\end{itemize}

It is not hard to verify that the modified chain $M'$ is well-defined, ergodic, and satisfies the aforementioned property (namely, $(v,mid)$ leaves to $(v,out)$ with probability $1$ and is the only state that transits to $(v,out)$). It remains to show that $M'$ has mixing-time $O(T)$. Toward this goal, let us define yet another Markov chain $C$ that consists of three states $\{ in, mid, out\}$ with transition probability $C_{in,in} = C_{in,mid} = 1/2$, and $C_{mid,out} = C_{out,in} = 1$. Clearly, $C$ is ergodic and has constant mixing-time. Now, the key observation is that a random walk on $M'$ can be decomposed into walks on $M$ and $C$ in the following sense: every step on $M'$ corresponding to a step on $C$ in a natural way, and one step on $M'$ from $(u,out)$ to $(v,in)$ can be identified as a step from $u$ to $v$ in $M$. Note that the walks on $M$ and $C$ are independent, and in expectation, every $4$ steps of walk on $M'$ induce one step of walk on $M$. It is not hard to see from these observation that the mixing time of $M'$ is at most $8T$.

\section{The Bound When the Sum Is Less Than Mean}\label{asec:less}
We now prove the remaining part of Claim~\ref{claim:spectralchernoff}, i.e.

\begin{claim}\label{thm:mainmultiplicativebound2} Let $M$ be an ergodic Markov chain with state space $[n]$, stationary distribution $\pi$, and spectral expansion $\lambda = \lambda(M)$.  Let $(V_1,\dots, V_t)$ denote a $t$-step random walk on $M$ starting from an initial distribution $\varphi$ on $[n]$, i.e., $V_1 \leftarrow \varphi$.  For every $i \in [t]$, let $f_i: [n] \rightarrow [0,1]$ be a weight function at step $i$ such that the expected weight $\E_{v\leftarrow \pi}[f_i(v)] = \mu$ for all $i$. Define the total weight of the walk $(V_1,\dots, V_t)$ by $X \triangleq \sum_{i=1}^t f_i(V_i)$.  There exists some constant $c$ and a parameter $r >0$ that depends only on $\lambda$ and $\delta$ such that
\begin{eqnarray*}
\mbox{2. }\frac{\E[e^{-rX}]}{e^{-r(1-\delta)\mu t}} & \leq & c \|\varphi\|_\pi  \exp\left(-\delta^2  (1-\lambda)  \mu t/ 36\right) \quad  \quad \mbox{for $0 \leq \delta \leq 1$.}
\end{eqnarray*}
\end{claim}

We mimic the proof strategy presented in Section~\ref{sec:boundmgf}. Observe first that
$$\E[e^{-rX}] = \|xP_1MP_2...MP_t\|_1,$$
where $P_i$'s are diagonal matrices with diagonal entries $(P_i)_{j,j} \triangleq e^{-rf_i(j)}$ for $j \in [n]$. Thus, our goal is to bound the moment generating function
$\E[e^{rX}]$.

Similar to the analysis presented in Section~\ref{sec:boundmgf}, we need to understand the effect of the $P_i$ operators.

\begin{lemma}\label{lem:ndop}Let $M$ be an ergodic Markov chain with state space $[n]$ and stationary distribution $\pi$. Let $f: [n] \rightarrow [0, 1]$ be a weight function with $\E_{v \leftarrow \pi}[f(v)] = \mu$. Let $P$ be a diagonal matrix with diagonal entries $P_{j, j} \triangleq e^{-rf(j)}$ for $j \in [n]$, where $r$ is a parameter satisfying $0 \leq r \leq 1/2$. We have
\begin{itemize}
\item $\|(\pi P)^{\parallel}\|_{\pi} \leq 1 - r \mu + \frac{r^2}{2}\mu.$
\item $\|(\pi P)^{\bot}\|_{\pi} \leq \sqrt 2 r \sqrt{\mu}$
\item For every vector $y \bot \pi$, $\|(yP)^{\parallel}\|_{\pi} \leq r \sqrt{\mu}\|y\|_{\pi}$.
\item For every vector $y \bot \pi$, $\|(yP)^{\bot}\|_{\pi} \leq \|y\|_{\pi}$
\end{itemize}
\end{lemma}

\begin{proof}
For the first item, we have
\begin{eqnarray*}
\|(\pi P)^{\parallel}\|_{\pi} & = & \sum_{i \leq n}e^{-r f(i)}\pi_i \\
& \leq & \sum_{i \leq n}(1-r f(i) + \frac{r^2}{2}f(i))\pi_i  \\
& \leq & 1 - r \mu + \frac{r^2}{2}\mu
\end{eqnarray*}
The first inequality holds because  $e^{-rx} \leq 1 - rx + r^2x/2$ for $0 \leq x \leq 1$.

(2). we may use Pythagorean theorem and get
\begin{eqnarray*}
\|(\pi P)^{\bot}\|^2_{\pi} & = & \|(\pi P)\|^2_{\pi} - \|(\pi P)^{\parallel}\|^2_{\pi} \\
& = & \sum_{i \leq n}e^{-2r f(i)}\pi_i - \left(\sum_{i \leq n}e^{-r f(i)}\pi_i\right)^2 \\
& \leq & \sum_{i \leq n}\left(1 - 2rf(i) + 2r^2f^2(i)\right)\pi_i
- \left(\sum_{i \leq n}(1-rf(i))\pi_i\right)^2 \\
& = &2r^2 \mu - r^2 \mu^2 \\
& \leq &2r^2 \mu.
\end{eqnarray*}
This implies $\|(\pi P)^{\bot}\|_{\pi} \leq \sqrt 2 r \sqrt{\mu}$.

(3). First, since $y \bot \pi$, we have $\langle y, \pi \rangle_{\pi} = 0$.
Next notice that by Cauchy Schwarz inequality,
$$\|(yP)^{\parallel}\|_{\pi} = \langle y, \pi P \rangle_{\pi} - \langle y, \pi I\rangle_{\pi} =
\langle y, \pi (P - I)\rangle_{\pi} \leq \|y\|_{\pi} \|\pi(P - I)\|_{\pi}.$$
We next bound $\|\pi (P - I)\|_{\pi}$. Specifically,
\begin{eqnarray*}
\|\pi(P - I)\|^2_{\pi} & = & \sum_{i \leq n}(e^{-rf(i)} - 1)^2\pi_i \\
& = & \sum_{i \leq n}(1-e^{-rf(i)})^2 \pi_i \\
& \leq & \sum_{i \leq n}(r f(i))^2 \pi_i \\
& \leq & r^2 \sum_{i \leq} f(i) \pi_i \\
& \leq & r^2 \mu.
\end{eqnarray*}
Therefore, $\|(yP)^{\parallel}\|_{\pi} \leq r \sqrt \mu \|y\|_{\pi}$.

(4). We have $\|(yP)^{\bot}\|_{\pi} \leq \|(yP)\|_{\pi} \leq \|y\|_{\pi}$.
\end{proof}

Now we proceed to prove Claim~\ref{thm:mainmultiplicativebound2} using Lemma~\ref{lem:ndop}.

\begin{proof} (of Claim~\ref{thm:mainmultiplicativebound2}). 
Let us recall that $z_0 \triangleq x$ and $z_i = z_{i - 1}P_iM$ for $i \in [t]$. Lemma~\ref{lem:ndop} gives us
$$\|z^{\parallel}_{i}\|_{\pi} \leq (1 - r \mu + \frac{r^2}{2}\mu) \|z^{\parallel_{i - 1}}\|_{\pi} + r\sqrt{\mu}\|z^{\bot}_{i - 1}\|_{\pi}$$
and
$$\|z^{\bot}_i\|_{\pi} \leq  \sqrt 2 \lambda r \sqrt \mu \| z^{\parallel}_{i - 1}\|_{\pi}+ \lambda \|z^{\bot}_{i - 1}\|_{\pi}$$
Following our strategy presented in Section~\ref{sec:boundmgf}, let $\alpha_0 = \|z^{\parallel}_0\|_{\pi} = 1$ and $\beta_0 = \|z^{\bot}_0\|_{\pi}$ and define for each $i \in [t]$,
\begin{equation}\label{eqn:ndalpha}
\alpha_i = (1-r\mu + \mu r^2/2)\alpha_{i - 1} + r \sqrt{\mu}\beta_{i - 1}
\end{equation}
and
\begin{equation}\label{eqn:ndbeta}
\beta_i = (\sqrt 2r \lambda\sqrt{\mu}) \alpha_{i - 1} + \lambda \beta_{i - 1}.
\end{equation}
We can inductively show that $\|z^{\parallel}_i\|_{\pi} \leq \alpha_i$ and $\|z^{\bot}_i\|_{\pi} \leq \beta_i$ for each $i \in [t]$.

Our goal becomes to give an upper bound for $\alpha_i$ and $\beta_i$. Also, we shall set
 $r = \min\{1/2,\log(1/\lambda)/2, 1-\sqrt \lambda, (1-\lambda)\delta/8\}$ throughout our analysis.
Next, we recursively substitute the value of $\beta_i$ from Eq.(\ref{eqn:ndbeta}) into Eq.(\ref{eqn:ndalpha}) and yield,
\begin{equation}\label{eqn:ndalphabnd}
\alpha_i = (1-(r-r^2/2)\mu)\alpha_{i - 1} + \sqrt{2}r^2 \mu  \lambda \alpha_{i - 2} + ... + \sqrt 2 r^2  \mu \lambda^{i - 1}\alpha_0 + r \sqrt{\mu}\lambda^{i - 1}\beta_0
\end{equation}
Using the fact that $r \leq 1 - \sqrt \lambda$ and thus $\alpha_{i} \leq (1-(r-r^2/2)\mu)\alpha_{i - 1}$ for all $i \geq 1$, we may conclude $\sqrt \lambda \alpha_{i - 1} \leq \alpha_i$. Now (\ref{eqn:ndalphabnd}) becomes
\begin{equation}
\alpha \leq \left(1-(r-r^2/2)\mu +\sqrt 2 r^2 \sqrt{\mu} \left(\sum_{j = 1}^{i - 1}\sqrt{\mu} \sqrt{\lambda^{i - j}}\right) \right)\alpha_{i - 1} + r\sqrt{\mu}\lambda^{i - 1}\beta_0.
\end{equation}
Next, let us define $A_i$ as follows,
$$
A_i \triangleq \left(1 - (r-r^2/2)\mu + \sqrt 2 r^2 \sqrt{\mu}\left(\sum_{j = 0}^{i - 1}\sqrt{\mu} \sqrt{\lambda^{i - j}}\right)\right). $$
We then have
$$\alpha_i \leq A_i \alpha_{i - 1} + r\sqrt{\mu}\lambda^{i - 1}\beta_0.$$
Therefore, we can see that
$$\alpha_t \leq \left(\prod_{i \leq t}A_i\right)\left(1+\beta_0\frac{r\sqrt{\mu}}{1-\lambda}\right).
$$
On the other hand, we can see that
\begin{eqnarray*}
\left(\prod_{1\leq i \leq t}A_i\right) & \leq & \exp\left\{\sum_{i\leq t} \left(
-(r-r^2/2)\mu\right) + \sum_{1\leq i \leq t} \sqrt 2 r^2\sqrt{\mu}\left(\sum_{1 \leq j \leq t-1}\sqrt{\lambda^{i - j} \mu}\right)\right\} \\
& \leq & \exp\left\{-(r-r^2/2) \mu t + \frac{2\sqrt 2 r^2}{1-\lambda}\mu t\right\} \\
& = & \exp\left\{-r\mu t + \left(\frac{r^2}{2} + \frac{2\sqrt 2 r^2}{1-\lambda}\right)\mu t\right\} \\
& \leq & \exp\left\{-r\mu t + \left(\frac{4r^2}{1-\lambda} \mu t\right)\right\}\\
\end{eqnarray*}
Notice that $1+\beta_0\frac{r\sqrt{\mu}}{1-\lambda} = O\left(\frac{r\|x\|_{\pi}}{1-\lambda}\right)$.
By using the fact $r = \min\{1/2,\log(1/\lambda)/2, 1-\sqrt \lambda, (1-\lambda)\delta/8\}$, we complete the proof.
\end{proof}

\section{Continuous Time Case} \label{sec:continuous time appendix}
This section proves Theorem~\ref{thm:continousbound}.

\begin{proof} (of Theorem~\ref{thm:continousbound}). We mimic the strategy from Lezaud~\cite{Lezaud04} to discretize the chain in $b$ time units, i.e. consider the states $v_{ib}$ for $i = 0, 1, ..., t/b$. The stationary distribution of this discretized chain $v_{ib}$ is the same as the original continuous time chain, and hence $\mu = \E_{\pi}f_t(v_t) = \E_{\pi}f_{ib}(v_{ib})$. Now by Theorem~\ref{thm:mixdeviation} we have
\begin{eqnarray*}
\mbox{1. } \Pr\left[\sum_{i = 1}^{t/b}f_{ib}(v_{ib})\geq (1+\delta)\frac{\delta\mu t} b\right] & \leq &
\begin{cases}
c \|\varphi\|_\pi\exp\left(-\delta^2 \mu (t/b) / (72T/b)\right) & \mbox{ for $0 \leq \delta \leq 1$} \\
c\|\varphi\|_\pi\exp\left(-\delta \mu (t/b)/(72T/b)\right) & \mbox{ for $\delta > 1$}
\end{cases}
\\
\mbox{2. }\Pr\left[\sum_{i = 1}^{t/b}f_{ib}(v_{ib})\leq (1-\delta)\frac{\delta\mu t} b\right] & \leq & c \|\varphi\|_\pi \exp\left(-\delta^2\mu (t/b)/ (72T/b)\right) \quad  \quad \mbox{for $0 \leq \delta \leq 1$}
\end{eqnarray*}
Notice that the mixing time for the discretized chain is $T/b$ while the total number of steps here is $t/b$. In the exponents, the term $b$ appears in both the numerator and the denominator and they cancel with each other. Taking limit as $b\rightarrow 0$ completes the proof~\cite{Lezaud04}.
\end{proof}

\end{document}